\documentclass[a4paper,reqno,10pt]{amsart}

\usepackage[utf8]{inputenc}
\usepackage[english]{babel}
\usepackage{amssymb}
\usepackage{hyperref}

\newcommand{\gl}{\mathfrak{gl}}
\newcommand{\ssl}{\mathfrak{sl}}
\newcommand{\so}{\mathfrak{so}}
\newcommand{\su}{\mathfrak{su}}

\DeclareMathOperator{\Exp}{Exp}
\DeclareMathOperator{\Iso}{Iso}
\DeclareMathOperator{\SL}{SL}
\DeclareMathOperator{\sspan}{span}
\DeclareMathOperator{\SU}{SU}

\theoremstyle{plain}
\newtheorem{theorem}{Theorem}[section]

\newtheorem{proposition}[theorem]{Proposition}
\newtheorem{lemma}[theorem]{Lemma}
\newtheorem*{utheorem}{Theorem}

\theoremstyle{definition}

\newtheorem{example}[theorem]{Example}

\theoremstyle{remark}
\newtheorem{remark}[theorem]{Remark}

\numberwithin{equation}{section}

\begin{document}

\title[The index of symmetry of $3$-dimensional Lie groups]{The index of symmetry of three-dimensional Lie groups with a left-invariant metric}

\author{Silvio Reggiani}

\address{CONICET and Universidad Nacional de Rosario, Dpto. de Matemática, ECEN-FCEIA, Av.\ Pellegrini 250, 2000 Rosario, Argentina }

\email{reggiani@fceia.unr.edu.ar}

\urladdr{\url{http://www.fceia.unr.edu.ar/~reggiani}}

\date{\today}

\thanks{Supported by CONICET. Partially supported by ANPCyT and SeCyT-UNR}



\begin{abstract}
  We determine the index of symmetry of $3$-dimensional unimodular Lie groups with a left-invariant metric. In particular, we prove that every $3$-dimensional unimodular Lie group admits a left-invariant metric with positive index of symmetry. We also study the geometry of the quotients by the so-called foliation of symmetry, and we explain in what cases the group fibers over a $2$-dimensional space of constant curvature.
\end{abstract}

\maketitle

\section{Introduction}

The index of symmetry of a Riemannian manifold $M$ is a geometric invariant which measures how far is $M$ from being a symmetric space. Informally, it can be defined as the infimum, over $p \in M$, of the maximum number of independent directions in $T_pM$ arising from Killing fields parallel at $p$. That is, if the index of symmetry of $M$ is $k$, then for each $p \in M$ there exist at least $k$ Killing fields $X_1, \ldots, X_k$ such that $X_1(p), \ldots, X_k(p)$ are linear independent in $T_pM$ and $(\nabla X_i)_p = 0$ for all $i = 1, \ldots, k$. So, the index of symmetry of $M$ equals $\dim M$ if and only if $M$ is a symmetric space. This notion appears first in \cite{olmos-reggiani-tamaru-2014} and it turned out to be an interesting way of studying homogeneous spaces. Most of the work was done in the compact case (see \cite{olmos-reggiani-tamaru-2014, berndt-olmos-reggiani-2015, podesta-2015}) including the classification of compact spaces with low co-index of symmetry.

The purpose of this article is to compute the index of symmetry of $3$-dimensional unimodular Lie groups with a left-invariant metric. Recall that, up to isomorphism, there are six $3$-dimensional simply connected, unimodular Lie groups: the abelian Lie group $\mathbb R^3$; $\SU(2)$ which is diffeomorphic to a $3$-dimensional sphere, the universal covering $\widetilde{\SL}(2, \mathbb R)$ of the real special linear group of degree $2$; the $3$-dimensional Heisenberg Lie group $H_1$; the universal covering $\tilde E(2)$ of the group of rigid motions of the euclidean plane; and the group of rigid motions $E(1, 1)$ of the Minkowski $2$-space. These groups are of a particular importance since most of the Thurston geometries ---except $\mathbb S^2 \times \mathbb R$, $\mathbb H^2 \times \mathbb R$ and $\mathbb H^3$--- can be realized as a left-invariant metric on a unimodular $3$-dimensional group. Since all the left-invariant metrics on $\SU(2)$ with non-trivial index of symmetry were determined in \cite{berndt-olmos-reggiani-2015}, we direct our attention to the non-compact unimodular Lie groups. As a consequence of our study we can state the following result.

\begin{utheorem}
  Every $3$-dimensional unimodular Lie group admits a left-invariant metric with positive index of symmetry.
\end{utheorem}

The last part of the article concerns the geometry of the so-called foliation of symmetry. The leaves of the foliation of symmetry $\mathcal L$ of the Riemannian manifold $M$ consist of totally geodesic submanifolds which are extrinsically isometric to a symmetric space and of maximal dimension (this is defined precisely in Section~\ref{sec:preliminaries}). We prove that when $M$ is a $3$-dimensional unimodular Lie group, the foliation of symmetry induces a Riemannian submersion $M \to M/\mathcal L$ if and only if $M$ is $\SU(2)$, $\widetilde\SL(2, \mathbb R)$ or $H_1$ and the metric is naturally reductive. For the last two cases, we recover the line bundles $\widetilde\SL(2, \mathbb R) \to \mathbb H^2$ over the hyperbolic plane, an the nontrivial vector bundle $H_1 \to \mathbb R^2$. It is also worth mentioning that for these naturally reductive metrics, the foliation of symmetry is given by the set of fixed points of the full isotropy group. Thus, these spaces share this important property with compact naturally reductive spaces (see \cite{olmos-reggiani-tamaru-2014}). 

Finally, we want to point out that the article includes, in Section \ref{sec:preliminaries}, a general result which says that the distribution of symmetry of a homogeneous space cannot be of codimension $1$. This theorem requires some delicate arguments, including an identification of Killing fields on a Riemannian manifold $M$ with parallel sections of the associated vector bundle $TM \oplus \Lambda^2(TM)$ (see \cite{kostant-1955,console-olmos-2008}).

\subsection*{Acknowledgement} The author wants to thank Carlos Olmos for his very useful comments and suggestions on the manuscript, specially the inclusion of Theorem~\ref{sec:ind=1}.

\section{Preliminaries}\label{sec:preliminaries}

In this section we essentially present the definitions and basic results concerning the index of symmetry of an homogeneous manifold. We refer to \cite{olmos-reggiani-tamaru-2014} and \cite{berndt-olmos-reggiani-2015} for more details and some structure theory for compact manifolds. Let $M = G/H$ be a Riemannian homogeneous manifold, we say that $M$ has \emph{index of symmetry} $i_{\mathfrak s}(M) = k$ if for (any) $p \in M$ there exist Killing fields $X_1, \ldots, X_k$ such that $X_1(p), \ldots, X_k(p)$ are linear independent in $T_pM$, $(\nabla X_1)_p = \cdots = (\nabla X_k)_p = 0$ and $k$ is the greatest integer with such properties. In particular, $i_{\mathfrak s}(M) = \dim M$ if and only if $M$ is a symmetric space. The so-called \emph{distribution of symmetry} $\mathfrak s$ of $M$ is defined as
$$\mathfrak s_p = \{v \in T_pM: \text{there exists a Killing field $X$ with $X_p = v$ and $(\nabla X)_p = 0$}\}.$$

Note that $\mathfrak s$ is $G$-invariant and moreover, is invariant under the full isometry group $\Iso(M)$ of $M$. We have that $\mathfrak s$ is autoparallel (i.e., integrable with totally geodesic leaves). The integral manifold $L(p)$ of $\mathfrak s$ by $p$ is called the \emph{leaf of symmetry} by $p$. It is known that $L(p)$ is extrinsically isometric to a symmetric space. The \emph{foliation of symmetry} $\mathcal L$ of $M$ consists of all the leaves of symmetry of $M$.

\begin{example}
  [\cite{olmos-reggiani-tamaru-2014}] If $M = G/H$ is a compact locally irreducible naturally reductive space, and $M$ is not a symmetric space, then the distribution of symmetry of $M$ is the $G$-invariant distribution induced by the set of fixed vectors of $H$ in the tangent space (via the isotropy representation). In this case, the leaf of symmetry at a given point is a symmetric space of the group type.
\end{example}

In \cite{berndt-olmos-reggiani-2015} it is proved that $i_{\mathfrak s}(M) \le \dim M - 2$ whenever $M$ is a homogeneous compact space which is not symmetric. We adapt the argument to noncompact spaces and, in particular, to our case of interest (even to non-unimodular $3$-dimensional groups).

\begin{theorem}\label{sec:ind=1}
  Let $M$ be a simply connected homogeneous Riemannian manifold of dimension $n$, then $i_{\mathfrak s}(M) \neq n - 1$.
\end{theorem}

The proof of this theorem requires a general framework to deal with Killing fields.

\begin{remark}\label{killing-fields}
  Killing fields on a Riemannian manifold $M$ can be regarded as parallel sections of the canonical associated vector bundle $TM \oplus \Lambda^2(TM) \simeq TM \oplus \so(TM)$, with respect to the connection $\bar\nabla$ on $TM \oplus \so(TM)$ given by
  \[
  \bar\nabla_X(Y, B) = (\nabla_XY - BX, \nabla_XB - R_{X, Y}),
  \]
  where $Y$ is a tangent field on $M$, $B$ is a skew-symmetric tensor on $M$ of type $(1, 1)$ and $\nabla$ is the Levi-Civita connection on $M$. In fact, if $(Y, B)$ is a parallel section, then $\nabla Y = B$ and so $Y$ is a Killing field on $M$. Conversely, if $Y$ is a Killing field on $M$, then its canonical lift $(Y, \nabla Y)$ is a parallel section of $TM \oplus \so(TM)$, since the derivative of $Y$ satisfies the so-called affine Killing equation $\nabla_X(\nabla Y) = R_{X, Y}$ (see for instance \cite{console-olmos-2008}).
\end{remark}

\begin{lemma}\label{bracket-killing-fields}
  Let $X, X'$ be two Killing fields on $M$ with initial conditions at $p \in M$ given by $(v, B)$, $(v', B') \in T_pM \oplus \so(T_pM)$ respectively. Then $[X, X']$ has initial conditions
  \[
  (B'v - Bv', R_{v, v'} - [B, B']).
  \]
\end{lemma}

\begin{proof}
  Since $B = (\nabla X)_p$ and $B' = (\nabla X')_p$, the first initial condition holds trivially. For the second one, we use the affine Killing equation (see Remark \ref{killing-fields}) and Bianchi first identity. In fact, if $w \in T_pM$ then
  \begin{align*}
    \nabla_w[X, X'] & = \nabla_w(\nabla_XX' - \nabla_{X'}X) = \nabla_w(\nabla_XX') - \nabla_w(\nabla_{X'}X) \\
    & = (\nabla_w(\nabla X'))_v + \nabla_{\nabla_w X}X' - (\nabla_w(\nabla X))_{v'} - \nabla_{\nabla_wX'}X \\
    & = R_{w, v'}v + B'Bw - R_{w, v}v' - BB'w \\
    & = R_{v, v'}w - [B, B']w = (R_{v, v'} - [B, B'])w. \qedhere
  \end{align*}
\end{proof}

\begin{proof}
  [{\proofname} of Theorem \ref{sec:ind=1}]
  Assume that the distribution of symmetry of $M$ has dimension $\dim \mathfrak s = n - 1$. So, given $p \in M$, we can choose $v \in T_pM$ and a vector field $X$ on $M$ which is always tangent to $\mathfrak s$ and such that $\nabla_vX \notin \mathfrak s_p$. Recall that $(X, 0)$ is a section of the $\bar\nabla$-parallel and flat subbundle $\mathfrak s \oplus \so(\mathfrak s)$ of the canonical bundle $TM \oplus \so(TM)$. Since
  \[
  \bar\nabla_v(X, 0) = (\nabla_vX, -R_{v, X(p)}) \in \mathfrak s_p \oplus \so(\mathfrak s_p),
  \]
  there exists a Killing field $Z$ on $M$ with initial conditions $Z(p) = \nabla_vX$ and $(\nabla Z)_p = -R_{v, X(p)}$. Moreover, we can assume that $Z(p)$ is orthogonal to $\mathfrak s_p$. In fact, let $w$ be the orthogonal projection of $\nabla_vX$ to $\mathfrak s_p$ and subtract to $Z$ the Killing field with initial conditions $(w, 0)$. We can also assume that $v$ is orthogonal to $\mathfrak s_p$, otherwise we subs-tract to $Z$ the Killing field with initial conditions $(0, R_{u, X(p)})$, where $u$ is orthogonal projection of $v$ to $\mathfrak s_p$. (Such a Killing field exists, since from Lemma~\ref{bracket-killing-fields}, the bracket $[U, X']$ of the Killing fields $U$ and $X'$ with initial conditions $(u, 0)$ and $(X(p), 0)$ respectively, has this property.)

  We have that $Z$ is always perpendicular to the leaf of symmetry $L(p)$ by $p$. In fact, let $\bar Z$ be the projection to $L(p)$ of the restriction of $Z$ to $L(p)$. Then $\bar Z$ is a Killing field of $L(p)$ with initial conditions $(0, (\nabla\bar Z)_p)$. Since $L(p)$ is totally geodesic, we have that
  \[
  \langle(\nabla\bar Z)_pv_1, v_2\rangle = - \langle R_{v, X(p)}v_1, v_2\rangle = 0
  \]
  for all $v_1, v_2 \in T_pL(p) = \mathfrak s_p$. So, the initial conditions of $\bar Z$ both vanish and $\bar Z = 0$. Therefore $Z$ is always perpendicular to $L(p)$ and moreover, $Z$ must be always perpendicular to the distribution of symmetry $\mathfrak s$. If fact, $\mathfrak s$ is invariant by isometries and the flow of $Z$ applied to $L(p)$ is open, since $\dim \mathfrak s = n - 1$.

  We also observe that the Killing field $Z$, orthogonal to the distributions of symmetry, is unique up to a constant multiple. In fact, let $Z'$ be a Killing field which is always perpendicular to $\mathfrak s$ and such that $Z'(p) = 0$. Let $\Phi_t'$ be the flow of $Z'$. Since $\Phi_t'(p) = p$, we have that $\Phi_t'(L(p)) = L(p)$, since $\Phi_t'$ preserves the foliation of symmetry. So, $Z'$ is always tangent to $L(p)$ and thus $Z'|_{L(p)} = 0$. Then $(\nabla Z')_p|_{\mathfrak s_p} = 0$ and, since $\mathfrak s_p$ has codimension $1$, it follows that $(\nabla Z')_p = 0$. This shows that $Z' = 0$, and so $Z$ is unique up to a constant multiple.

  It follows from the Killing equation that $\langle\nabla_ZZ, Z\rangle = 0$, and so $\nabla_ZZ$ lies in the distribution of symmetry. Let $\mathfrak d$ be the $2$-dimensional distribution spanned by $Z$ and $\nabla_ZZ$. We will show that $\mathfrak d$ is a parallel distribution. Using this fact, we conclude that $M$ splits as a $2$-dimensional symmetric by a $(n - 1)$-symmetric space, since $\mathfrak d \cap \mathfrak s = \mathbb R\nabla_ZZ$ is a parallel distribution when restricted to $L(p)$, and this is a contradiction.

  In order to show that $\mathfrak d$ is a parallel distribution, we note that $\mathfrak d$ is invariant under isometries, since $Z$ is unique up to a scalar multiple. This implies that the $1$-dimensional distribution $\mathfrak d \cap \mathfrak s$ is parallel along any leaf of symmetry $L(q)$, $q \in M$. In fact, if $U$ is a Killing field with initial conditions at $q$, given by $(u, 0)$ with $u \in \mathfrak s_q$, then $[U, \nabla_ZZ]$ must be a (functional) multiple of $\nabla_ZZ$, and therefore $\mathfrak d \cap \mathfrak s$ is parallel along $L(q)$. Since $L(q)$ is totally geodesic and $\mathfrak d^\bot \subset \mathfrak s$, we conclude that $\mathfrak d^\bot$ is autoparallel. Let us prove now that $\mathfrak d$ is also autoparallel. Normalize $Z$ and $\nabla_ZZ$ in order to get two unit vector fields $Y_1, Y_2$ spanning $\mathfrak d$. So, $Y_1$ is orthogonal to $\mathfrak s$ and $Y_2$ lies on $\mathfrak s$. Moreover, the integral lines of $Y_2$ are geodesics (always tangent to the distribution of symmetry), and thus $\nabla_{Y_2}Y_2 = 0$. Since $\nabla_ZZ$ is proportional to $Y_2$, we also have that $\nabla_{Y_1}Y_1$ lies in $\mathfrak d$. Observe now that $Y_1$ is a unit vector field normal to any totally geodesic hypersurface $L(q)$. So, $Y_1$ must be parallel along any geodesic of $L(q)$. In particular, $\nabla_{Y_2}Y_1 = 0$. This implies that $\nabla_{Y_1}Y_2 = [Y_1, Y_2]$, which lies in $\mathfrak d$. In fact, $[Z, Y_2] = 0$ since $\mathbb RY_2 = \mathfrak d \cap \mathfrak s$ is invariant under isometries. So $\mathfrak d$ and $\mathfrak d^\bot$ are both autoparallel and this implies that $\mathfrak d$ is a parallel distribution.
\end{proof}

\begin{remark}
  We will use repeatedly the well-known Koszul formula in Killing fields (see for instance \cite{berndt-olmos-reggiani-2015}): if $M$ is a Riemannian manifold and $X^*, Y^*, Z^*$ are Killing fields on $M$, then
\begin{equation}\label{eq:koszul-right-invariant}
  \langle\nabla_{X^*}Y^*, Z^*\rangle = \frac{1}{2}\left(\langle[X^*, Y^*], Z^*\rangle + \langle[X^*, Z^*], Y^*\rangle + \langle[Y^*, Z^*], X^*\rangle\right).
\end{equation}

In the case of a Lie group it is also useful the Koszul formula on left-invariant fields: if $G$ is a Lie group endowed with a left-invariant metric and $X, Y, Z$ are left-invariant fields on $G$, then
\begin{equation}\label{eq:koszul-left-invariant}
  \langle\nabla_XY, Z\rangle = \frac{1}{2}\left(\langle[X, Y], Z\rangle - \langle[X, Z], Y\rangle - \langle[Y, Z], X\rangle\right).
\end{equation}
\end{remark}

\section{The case of the special unitary group $\SU(2)$}

Let us consider for the special unitary Lie group $\SU(2)$, the basis of its Lie algebra $\su(2) = \{A \in \gl(2, \mathbb C): A + A^* = 0\}$ given by
\begin{align*}
  X_1 = \frac{1}{2}
  \begin{pmatrix}
    i & 0 \\
    0 & -i
  \end{pmatrix},
  & & X_2 = \frac{1}{2}
  \begin{pmatrix}
    0 & -1 \\
    1 & 0
  \end{pmatrix},
  & & X_3 = \frac{1}{2}
  \begin{pmatrix}
    0 & -i \\
    -i & 0
  \end{pmatrix},
\end{align*}
where $i = \sqrt{-1}$. Any left-invariant metric on $\SU(2)$ is identified with a scalar product on $\su(2)$. Moreover, it follows from \cite{ha-lee-2009} that any left-invariant metric is automorphically isometric to one of the form $g_{\lambda, \mu, \nu}$, which is represented in the basis $X_1, X_2, X_3$ by the symmetric positive-definite matrix
$$
\begin{pmatrix}
  \lambda & 0 & 0 \\
  0 & \mu & 0 \\
  0 & 0 & \nu
\end{pmatrix}
$$
where $\lambda \ge \mu \ge \nu > 0$. In particular, such a metric is bi-invariant if and only if $\lambda = \mu = \nu$. In \cite{berndt-olmos-reggiani-2015}, the left-invariant metrics on $\SU(2)$ with non-trivial index of symmetry are computed. Moreover, this family classifies all the simply connected, compact, homogeneous manifolds with co-index of symmetry equals to~$2$.

\begin{theorem}
  [See \cite{berndt-olmos-reggiani-2015}]
  Let $g$ be a left-invariant metric on $\SU(2)$. Then
  \begin{enumerate}
  \item $i_{\mathfrak s}(\SU(2), g) = 3$ if and only if $g$ is bi-invariant, and in this case $(\SU(2), g)$ is isometric to a round sphere;
  \item $i_{\mathfrak s}(\SU(2), g) = 1$ if and only if, up to isometry and scaling, $g = g_{\lambda, \lambda - 1, 1}$ with $\lambda > 2$, $g = g_{\lambda, 1, 1}$ with $\lambda > 1$, or $g = g_{1, 1, \nu}$, with $0 < \nu < 1$.
  \end{enumerate}
\end{theorem}

\begin{remark}
  Recall that from \cite{berndt-olmos-reggiani-2015} (see also \cite{ha-lee-2012}),  all the isometry groups of the metrics of the family $g_{\lambda, \lambda - 1, 1}$ have dimension $3$. So, the symmetric subspace (inside the isometry algebra) for these metrics is spanned for a right-invariant vector field. In contrast, the metrics $g_{\lambda, 1, 1}$ and $g_{1, 1, \nu}$ have all $4$-dimensional isometry group, and the symmetric subspaces are never spanned by right-invariant Killing fields except for $g_{2, 1, 1}$, which correspond to unit tangent bundle of the $2$-dimensional sphere of curvature $2$ with the Sasaki metric (see \cite{olmos-reggiani-tamaru-2014}).
\end{remark}

\section{The case of special linear group $\SL(2, \mathbb R)$}\label{sec:case-sl2}

Since the distribution of symmetry lifts to the universal covering, we can perform our calculations on $\SL(2, \mathbb R)$ instead of $\widetilde\SL(2, \mathbb R)$. Let us consider the basis of the Lie algebra $\ssl(2, \mathbb R) = \{A \in \gl(2, \mathbb R): A + A^t = 0\}$ given by 
\begin{align*}
  X_1 = 
  \begin{pmatrix}
    0 & 1 \\
    -1 & 0
  \end{pmatrix}, & & 
  X_2 = 
  \begin{pmatrix}
    0 & 1 \\
    1 & 0
  \end{pmatrix}, & & 
  X_3 = 
  \begin{pmatrix}
    1 & 0 \\
    0 & -1
  \end{pmatrix},
\end{align*}
and note that $[X_1, X_2] = 2X_3$, $[X_1, X_3] = -2X_2$ and $[X_2, X_3] = -2X_1$. We have from \cite{ha-lee-2009} that any left-invariant metric on $\SL(2, \mathbb R)$ is represented, up to isometric automorphism, in the basis $X_1, X_2, X_3$ by the symmetric positive-definite matrix
$$
\begin{pmatrix}
  \lambda & 0 & 0 \\
  0 & \mu & 0 \\
  0 & 0 & \nu
\end{pmatrix},
$$
where, $\lambda > 0$ and  $\mu \ge \nu > 0$. Let us denote by $X_1^*, X_2^*, X_3^*$ the right-invariant vector fields with the same initial conditions as $X_1, X_2, X_3$. We consider separately two different geometric cases.

\subsection{Case $\mu > \nu$}

Here, by the results in \cite{ha-lee-2012}, the group of isometries of $g_{\lambda, \mu, \nu}$ has dimension $3$, and so any Killing field is right-invariant. We are looking for a nontrivial Killing field
$$Y^* = aX_1^* + bX_2^* + cX_3^*$$
such that $(\nabla Y^*)_e = 0$. Taking into account that $[X_i^*, X_j^*] = -[X_i, X_j]^*$ we get that
\begin{align}\label{eq:X_iY-SL2}
  [X_1^*, Y^*] & = 2(cX_2^* - bX_3^*), \notag \\ 
  [X_2^*, Y^*] & = 2(cX_1^* + aX_3^*), \\
  [X_3^*, Y^*] & = -2(bX_1^* + aX_2^*). \notag
\end{align}
We denote, for the sake of simplicity, $g_{\lambda, \mu, \nu} = \langle\cdot, \cdot\rangle$. Then, the condition $(\nabla Y^*)_e = 0$ is equivalent to $\langle\nabla_{X_i^*}Y^*, X_j^*\rangle_e = 0$ for $i < j$. By making use of (\ref{eq:koszul-right-invariant}) and (\ref{eq:X_iY-SL2}), we obtain
\begin{align*}
  0 & = 2\langle\nabla_{X_1^*}Y^*, X_2^*\rangle_e \\
  & = 2c\langle X_2^*, X_2^*\rangle_e - 2c\langle X_3^*, X_3^*\rangle_e - 2c\langle X_1^*, X_1^*\rangle_e \\
  & = 2c(-\lambda + \mu - \nu).
\end{align*}
Similarly,
\begin{align*}
  0 & = 2\langle\nabla_{X_1^*}Y^*, X_3^*\rangle_e \\
  & = -2b\langle X_3^*, X_3^*\rangle_e + 2b\langle X_2^*, X_2^*\rangle_e + 2b\langle X_1^*, X_1^*\rangle_e \\
  & = 2b(\lambda + \mu - \nu)
\end{align*}
and
\begin{align*}
  0 & = 2\langle\nabla_{X_2^*}Y^*, X_3^*\rangle_e \\
  & = 2a\langle X_3^*, X_3^*\rangle_e + 2a\langle X_1^*, X_1^*\rangle_e + 2a\langle X_1^*, X_1^*\rangle_e \\
  & = 2a(\lambda + \mu + \nu)
\end{align*}
So, $(\nabla Y^*)_e = 0$ if and only if
\begin{equation}\label{eq:1}
  2c(-\lambda + \mu - \nu) = 2b(\lambda + \mu - \nu) = 2a(\lambda + \mu + \nu) = 0.
\end{equation}
Since $\mu > 0$, it must be $a = b = 0$ and we have the non-trivial solution $Y^* = cX_3^*$, $c \neq 0$, if and only if $\lambda = \mu - \nu$.

\subsection{Case $\mu = \nu$}

Recall that from \cite{ha-lee-2012}, the dimension of the full isometry group of $g_{\lambda, \mu, \mu}$ is $4$. Moreover, the Lie algebra of $\Iso(\SL(2, \mathbb R), g_{\lambda, \mu, \mu})$ is $\ssl(2, \mathbb R) \rtimes \so(2)$, where the isotropy algebra $\so(2)$ acts only on $\sspan_{\mathbb R}\{X_2^*, X_3^*\}$. So, there exists a Killing field $X_4^*$ (the same for all $\lambda, \mu)$ such that $X_4^*(e) = 0$ and
\begin{align*}
  [X_1^*, X_4^*] = 0, & & [X_2^*, X_4^*] = -X_3^*, & & [X_3^*, X_4^*] = X_2^*.
\end{align*}

Let us denote, as in the previous subsection, $g_{\lambda, \mu, \mu} = \langle\cdot, \cdot\rangle$. Now, we are looking for a Killing field
$$Y^* = aX_1^* + bX_2^* + cX_3^* + dX_4^*$$
such that $(\nabla Y^*)_e = 0$. Recall that with same calculations as in (\ref{eq:1}), we have that $d \neq 0$ if $Y^*$ is nontrivial. Moreover, since $X_4(e) = 0$, we have $(\nabla_{X_i^*}X_4)_e = [X_i^*, X_4^*]_e$. If we call $\tilde Y^* = Y^* - dX_4^*$, then $\tilde Y^*$ is right-invariant, and the formulas for $\langle\nabla_{X_i^*}\tilde Y^*, X_j^*\rangle_e$, $1 \le i < j \le 3$, are the same as in the previous subsection taking $\mu = \nu$. So, the equations $\langle\nabla_{X_i^*}Y^*, X_j^*\rangle_e = 0$ become
\begin{align*}
  0 & = 2\langle\nabla_{X_1^*}Y^*, X_2^*\rangle_e = -2c\lambda \\
  0 & = 2\langle\nabla_{X_1^*}Y^*, X_3^*\rangle_e = 2b\lambda \\
  0 & = 2\langle\nabla_{X_2^*}Y^*, X_3^*\rangle_e = 2a(\lambda + 2\mu) -2d\mu.
\end{align*}
Thus, $b = c = 0$ and we have a nontrivial solution with $d = a(\lambda + 2\mu)/\mu$. In particular, if we take $a = 1$, then $Y^* = X_1^* + (\frac{\lambda}{\mu} + 2)X_4^*$.

We can summarize the results of this section with the following theorem.

\begin{theorem}
  Let $g$ be a left-invariant metric on $\SL(2, \mathbb R)$ such that $i_{\mathfrak s}(\SL(2, \mathbb R), g)$ is nontrivial. Then, up to isometry and scaling, $g = g_{\lambda, \lambda + 1, 1}$ or $g = g_{\lambda, 1, 1}$, where $\lambda > 0$.
\end{theorem}

\section{The case of the Heisenberg Lie group $H_1$}

Let us consider the $3$-dimensional Heisenberg Lie group modelled as the differentiable manifold $H_1 = \mathbb R^3$ with the group structure
$$(x, y, z)(x', y', z') = (x + x', y + y', z + z' + \tfrac{1}{2}(xy' - yx')).$$
Let us denote the identity element of $H_1$ by $e = (0, 0, 0)$. Recall that a basis of left-invariant vector field is given by
\begin{align}\label{eq:orthonormal-basis-H1}
  X_1 = \frac{\partial}{\partial x} - \frac{y}{2}\frac{\partial}{\partial z}, & & X_2 = \frac{\partial}{\partial y} + \frac{x}{2}\frac{\partial}{\partial z}, & & X_3 = \frac{\partial}{\partial z}
\end{align}
in canonical coordinates $(x, y, z)$. Notice that $X_3$ spans the center of the Lie algebra $\mathfrak h_1$ of $H_1$ and $[X_1, X_2] = X_3$. It is well-known that there exists only one left-invariant metric on $H_1$, up to isometry and scaling (see Remark \ref{sec:left-invariant-metric-H1}). So, in order to compute the index of symmetry of $H_1$ with respect to a left-invariant metric, we can assume that it is endowed with left-invariant metric $\langle\cdot, \cdot\rangle$ that makes (\ref{eq:orthonormal-basis-H1}) an orthonormal basis.

Notice that, since $H_1$ is not a symmetric space, according to Theorem \ref{sec:ind=1}, its index of symmetry must be $0$ or $1$. So we look for the solutions of $(\nabla Y^*)_e = 0$, where $Y^*$ is a Killing vector field on $H_1$. It is well-known that the full isometry group of $H_1$ has $1$-dimensional isotropy group. Moreover, the Lie algebra of the isotropy group is generated for the Killing field
$$X_4^* = -yX_1 + xX_2 - \frac{x^2 + y^2}{2}X_3,$$
and so any Killing vector field can be written as the sum of a right-invariant vector field and a constant multiple of $X_4^*$. Let us consider the basis of right-invariant vector fields $X_1^*, X_2^*, X_3^*$ given by
\begin{align*}
  X_1^* = X_1 + yX_3, & & X_2^* = X_2 -xX_3, & & X_3^* = X_3.
\end{align*}
Then we can write the Killing vector field $Y^*$ as a linear combination
\begin{equation}\label{eq:generalY*}
  Y^* = aX_1^* + bX_2^* + cX_3^* + d X_4^*
\end{equation}
for some $a, b, c, d, \in \mathbb R$. In order to study the equation $(\nabla Y^*)_e = 0$ it is convenient to use Koszul formula in left-invariant fields. 
It is easy to verify from (\ref{eq:koszul-left-invariant}) that
$$\nabla_{X_1}X_1 = \nabla_{X_2}X_2 = \nabla_{X_3}X_3 = 0,$$
$$\nabla_{X_1}X_2 = \frac{1}{2}X_3 = -\nabla_{X_2}X_1$$
and
\begin{align*}
  \nabla_{X_3}X_1 = \nabla_{X_1}X_3 = -\frac{1}{2}X_2, & & \nabla_{X_3}X_2 = \nabla_{X_2}X_3 = \frac{1}{2}X_1.
\end{align*}

Assume that $Y^*$ as in (\ref{eq:generalY*}) satisfies $(\nabla Y^*)_e = 0$. We have in particular that
$$0 = (\nabla_{X_3}Y^*)_e = a(\nabla_{X_3}X_1^*)_e + b(\nabla_{X_3}X_2^*)_e$$
since also at the origin we have $0 = [X_3, X_4^*]_e = (\nabla_{X_3}X_4^*)_e$. By replacing $X_1^*$ and $X_2^*$ in terms of left-invariant vector fields in the above equation we get that $0 = (b/2)X_1(e) - (a/2)X_2(e)$ and therefore $a = b = 0$. So, equation (\ref{eq:generalY*}) has now the form $Y^* = cX_3^* + dX_4^*$. Note that neither $c$ nor $d$ can be equal to zero, and then we can assume that $c = 1$. Finally, since $(\nabla_{X_1}X_4^*)_e = X_2(e)$ and $(\nabla_{X_2}X_4^*)_e = -X_1(e)$, we have that
$$(\nabla_{X_1}Y^*)_e = (\nabla_{X_2}Y^*)_e = (\nabla_{X_3}Y^*)_e = 0$$
if and only if $d = -1/2$.

\begin{theorem}
  Let $H_1$ be the $3$-dimensional Heisenberg Lie group endowed with the left-invariant metric such that the basis $X_1, X_2, X_3$ of $\mathfrak h_1$ as in (\ref{eq:orthonormal-basis-H1}) is orthonormal. Then $H_1$ has index of symmetry $i_{\mathfrak s}(H_1) = 1$. Moreover, the Killing vector field
  $$Y^* = X_3 - \frac{1}{2}X_4^*$$
  is parallel at the origin, where $X_4^* = -yX_1 + xX_2 -\frac{1}{2}(x^2 + y^2)X_3$.
\end{theorem}

As mentioned before, the statement of the above theorem is true for any left-invariant metric on $H_1$.

\begin{remark}\label{sec:left-invariant-metric-H1}
  It is shown in \cite{lauret-2003} and also in \cite{kodama-takahara-tamaru-2011}, that the only $3$-dimensional simply connected Lie groups  admitting only one left-invariant metric, up to isometry and scaling, are the abelian Lie group $\mathbb R^3$, the simply connected Lie group acting simply transitively on the hyperbolic space $\mathbb H^3$, and the $3$-dimensional Heisenberg Lie group $H_1$. For the sake of completeness, we give a proof of this fact for the last case. Keeping the notation of the above paragraphs, let $\langle\cdot, \cdot\rangle'$ be an arbitrary inner product on $\mathfrak h_1$. By rescaling the metric we can assume that the left-invariant vector field $X_3$ has norm one with respect to both of the corresponding left-invariant metrics. Call $\mathfrak v$ the vector subspace of $\mathfrak h_1$ spanned by $X_1, X_2$, which is the orthogonal complement to $\mathfrak z$ with respect to $\langle\cdot, \cdot\rangle$. Let $\mathfrak v'$ be the orthogonal complement to $\mathfrak z$ with respect to $\langle\cdot, \cdot\rangle'$ and let $X_1', X_2'$ be an orthonormal basis of $\mathfrak v'$. It is easy to see that the linear isometry from $(\mathfrak h_1, \langle\cdot, \cdot\rangle)$ onto $(\mathfrak h_1, \langle\cdot, \cdot\rangle')$ mapping $X_1, X_2, X_3$ into $X_1', X_2', X_3$ respectively, lifts to an isometry from $(H_1, \langle\cdot, \cdot\rangle)$ onto $(H_1, \langle\cdot, \cdot\rangle')$.

  Recall that the above argument fails in higher dimensions (see \cite{vukmirovic-2015}). In fact, given an inner product $\langle\cdot, \cdot\rangle$ on the $(2n + 1)$-dimensional Heisenberg Lie algebra $\mathfrak h_n$, we decompose as an orthogonal sum $\mathfrak h_n = \mathfrak v \oplus \mathfrak z$ and for every $Z \in \mathfrak z$ we have associated the element $j(Z) \in \so(\mathfrak v)$ given by $\langle j(Z)X, Y\rangle = \langle[X, Y], Z\rangle$. Such a map is a geometric invariant. For the standard $H$-type left-invariant metric on $H_n$, $j(Z)$ acts on $\mathfrak v$ as a multiple of the multiplication by $\sqrt{-1}$, once identified $\mathfrak v \simeq \mathbb C^n$. But this is not always the case for an arbitrary left-invariant metric when $n > 1$.
\end{remark}

\section{The case of $\tilde E(2)$}

We present the universal cover $\tilde E(2)$ of the Euclidean group $E(2)$, as the underlying manifold $\mathbb R^3$ with the group multiplication given by
\begin{align*}
  (x, y, z)(x', y', z') & = 
(x', y', z')\begin{pmatrix}
  \cos z & \sin z & 0 \\
  -\sin z & \cos z & 0 \\
  0 & 0 & 1
\end{pmatrix} +
  (x, y, z) \\
  & = (x + x'\cos z -y'\sin z, y + x'\sin z + y'\cos z, z + z').
\end{align*}

Let us consider the basis of left-invariant fields given by
\begin{align*}
  X_1 = \cos z \frac{\partial}{\partial x} + \sin z\frac{\partial}{\partial y}, & & X_2 = -\sin z\frac{\partial}{\partial x} + \cos z\frac{\partial}{\partial y}, & & X_3 = \frac{\partial}{\partial z}.
\end{align*}

Such a basis satisfies the bracket relations
\begin{align*}
  [X_1, X_2] = 0, & & [X_1, X_3] = -X_2, & & [X_2, X_3] = X_1
\end{align*}
and one has the associated right-invariant vector fields
\begin{align*}
  X_1^* = \frac{\partial}{\partial x}, & & X_2^* = \frac{\partial}{\partial y}, & & X_3^* = -y\frac{\partial}{\partial x} + x\frac{\partial}{\partial y} + \frac{\partial}{\partial z}.
\end{align*}

It is a well-known fact that the left-invariant metric on $\tilde E(2)$ that makes of $X_1, X_2, X_3$ an orthonormal basis is flat (see for instance \cite{milnor-1976}), and hence isometric to the Euclidean $3$-dimensional space. Moreover, one can see that any flat metric on $\tilde E(2)$ is obtained from this one rescaling by a positive multiple the metric on the abelian ideal generated by $X_1$ and $X_2$. The non-flat left-invariant metrics on $\tilde E(2)$ are classified, up to isometric automorphism, in \cite{ha-lee-2009}, and they are represented in the basis $X_1, X_2, X_3$ by the symmetric positive-defined matrix
$$
\begin{pmatrix}
  1 & 0 & 0 \\
  0 & \mu & 0 \\
  0 & 0 & \nu
\end{pmatrix}$$
where $0 < \mu < 1$ and $\nu > 0$. Denote such a metric by $g_{\mu, \nu}$. In \cite{ha-lee-2012}, the isometry groups of $g_{\mu, \nu}$ are computed and, it particular, it follows that the full isometry group of $g_{\mu, \nu}$ has always dimension $3$. So, in order to compute $i_{\mathfrak s}(\tilde E(2), g_{\mu, \nu})$, we have to decide if there exists a right-invariant field
$$Y^* = aX_1^* + bX_2^* + cX_3^*$$
such that $(\nabla Y^*)_e = 0$. As usual, we denote $g_{\mu, \nu} = \langle\cdot, \cdot\rangle$. Now, using Koszul formula on Killing fields (\ref{eq:koszul-right-invariant}) and the brackets
\begin{align*}
  [X_1^*, Y^*] = cX_2^*, & & [X_2^*, Y^*] = -cX_1^*, & & [X_3^*, Y^*] = bX_1^* - aX_2^*
\end{align*}
we must have
\begin{align*}
  0 & = 2\langle\nabla_{X_1^*}Y^*, X_2^*\rangle_e \\
  & = c\langle X_2^*, X_2^*\rangle_e + \langle [X_1^*, X_2^*], Y^*\rangle_e + c\langle X_1^*, X_1^*\rangle_e \\
  & = c(1 + \mu),
\end{align*}
and so $c = 0$, 
\begin{align*}
  0 & = 2\langle\nabla_{X_1^*}Y^*, X_3^*\rangle_e \\
  & = c\langle X_2^*, X_3^*\rangle_e + b\langle X_2^*, X_2^*\rangle_e - b\langle X_1^*, X_1^*\rangle_e \\
  & = b(\mu - 1),
\end{align*}
which implies $b = 0$, and finally
\begin{align*}
  0 & = 2\langle\nabla_{X_2^*}Y^*, X_3^*\rangle_e \\
  & = -c\langle X_1^*, X_3^*\rangle_e - a\langle X_1^*, X_1^*\rangle_e + a\langle X_2^*, X_2^*\rangle_e \\
  & = a(1 - \mu)
\end{align*}
says that $a = 0$. So, there is not a non-zero Killing field $Y^*$ such that $(\nabla Y^*)_e = 0$ and therefore the index of symmetry of $g_{\mu, \nu}$ is equal to $0$ for all $0 < \mu < 1$, $\nu > 0$.

\begin{theorem}
  Let $g$ be a left-invariant metric on $\tilde E(2)$. Then $g$ has non-trivial index of symmetry if and only if $g$ is flat.
\end{theorem}

\section{The case of $E(1,1)$}

We present the group $E(1, 1)$, of rigid motions of the Minkowski $2$-space, as the underlying manifold $\mathbb R^3$ endowed with the group multiplication
$$(x, y, z)(x', y', z') = (x + e^zx', y + e^{-z}y', z + z').$$
It is easy to see that the left-invariant fields
\begin{align*}
  X_1 = e^z\frac{\partial}{\partial x}, & & X_2 = e^{-z}\frac{\partial}{\partial y}, & & X_3 = \frac{\partial}{\partial z}
\end{align*}
form a basis of $\mathfrak e(1, 1)$ such that
\begin{align*}
  [X_1, X_2] = 0, & & [X_1, X_3] = -X_1, & & [X_2, X_3] = X_2.
\end{align*}
The right-invariant fields $X_i^*$ such that $X_i^*(e) = X_i(e)$ are given by
\begin{align*}
  X_1^* = \frac{\partial}{\partial x}, & & X_2^* = \frac{\partial}{\partial y}, & & X_3^* = x\frac{\partial}{\partial x} - y\frac{\partial}{\partial y} + \frac{\partial}{\partial z}
\end{align*}
From the results in \cite{ha-lee-2012}, we have that there exist, up to isometric automorphism, two families of left-invariant metrics on $E(1, 1)$, which are represented in the basis $X_1, X_2, X_3$ by the symmetric positive-definite matrices
\begin{align}\label{eq:metricsE11}
  \begin{pmatrix}
    1 & 0 & 0 \\
    0 & 1 & 0 \\
    0 & 0 & \nu
  \end{pmatrix}, & & 
  \begin{pmatrix}
    1 & 1 & 0 \\
    1 & \mu & 0 \\
    0 & 0 & \nu
  \end{pmatrix},
\end{align}
where $\nu > 0$ (in both cases) and $\mu > 1$. We denote these metrics by $g_\nu$ and $g_{\mu, \nu}$ respectively. It also follows from \cite{ha-lee-2012} that any left-invariant metric on $E(1, 1)$ has $3$-dimensional full isometry group. So, for any left-invariant metric on $E(1, 1)$, a generic Killing field has the form
$$Y^* = aX_1^* + bX_2^* + cX_3^*.$$
In order to study the equation $(\nabla Y^*)_e = 0$ it is useful to note that
\begin{align}\label{eq:bracketsE11}
  [X_1^*, Y^*] = cX_1^*, & & [X_2^*, Y^*] = -cX_2^*, & & [X_3^*, Y^*] = -aX_1^* + bX_2^*. 
\end{align}

We distinguish two separate cases according to the metrics above.

\subsection{Case $g_\nu$}

Assume that $(\nabla Y^*)_e = 0$ and denote $g_\nu = \langle\cdot, \cdot\rangle$. From (\ref{eq:koszul-right-invariant}) and (\ref{eq:bracketsE11}) we must have
\begin{align*}
  0 & = 2\langle\nabla_{X_1^*}Y^*, X_2^*\rangle_e \\
  & = c\langle X_1^*, X_2^*\rangle_e + \langle [X_1^*, X_2^*], Y^*\rangle_e + c\langle X_2^*, X_1^*\rangle_e \\
  & = 0,
\end{align*}
so, no condition is imposed on $c \in \mathbb R$. On the other hand, 
\begin{align*}
  0 & = 2\langle\nabla_{X_1^*}Y^*, X_3^*\rangle_e \\
  & = c\langle X_1^*, X_3^*\rangle_e + a\langle X_1^*, X_1^*\rangle_e + a\langle X_1^*, X_1^*\rangle_e \\
  & = 2a,
\end{align*}
which means $a = 0$, and
\begin{align*}
  0 & = 2\langle\nabla_{X_2^*}Y^*, X_3^*\rangle_e \\
  & = -c\langle X_2^*, X_3^*\rangle_e - b\langle X_2^*, X_2^*\rangle_e - b\langle X_2^*, X_2^*\rangle_e \\
  & = -2b,
\end{align*}
so $b = 0$. Thus, we conclude that $Y^* = X_3^*$ is a Killing field parallel at $e$, for any value of $\nu > 0$.

\subsection{Case $g_{\mu, \nu}$}

Assume again that $(\nabla Y^*)_e = 0$ and denote $g_{\mu, \nu} = \langle\cdot, \cdot\rangle$. This case is slightly different from the previous one since $X_1, X_2, X_3$ is not an orthogonal basis. More precisely, then non-zero inner products are $\langle X_1, X_1\rangle = \langle X_1, X_2\rangle = 1$, $\langle X_2, X_2\rangle = \mu$ and $\langle X_3, X_3\rangle = \nu$. Once again, using the Koszul formula for Killing fields and (\ref{eq:bracketsE11}) we have the equations
\begin{align*}
  0 & = 2\langle\nabla_{X_1^*}Y^*, X_2^*\rangle_e \\
  & = c\langle X_1^*, X_2^*\rangle_e + \langle [X_1^*, X_2^*], Y^*\rangle_e + c\langle X_2^*, X_1^*\rangle_e \\
  & = 2c,
\end{align*}
which implies $c = 0$,
\begin{align*}
  0 & = 2\langle\nabla_{X_1^*}Y^*, X_3^*\rangle_e \\
  & = c\langle X_1^*, X_3^*\rangle_e + \langle X_1^*, Y^*\rangle_e + \langle aX_1^* - bX_2^*, X_1^*\rangle_e \\
  & = a + b + a - b = 2a,
\end{align*}
so $a = 0$, and similarly
\begin{align*}
  0 & = 2\langle\nabla_{X_2^*}Y^*, X_3^*\rangle_e \\
  & = -c\langle X_2^*, X_3^*\rangle_e - \langle X_2^*, Y^*\rangle_e + \langle aX_1^* - bX_2^*, X_2^*\rangle_e \\
  & = - a - b + a - b = -2b,
\end{align*}
shows that $b = 0$. Therefore, all the metrics $g_{\mu, \nu}$ have index of symmetry equal to zero.

We summarize the results of this section as follows.

\begin{theorem}
  Let $g$ be a left-invariant metric on $E(1, 1)$. Then $g$ has non-trivial index of symmetry if and only if $g = g_\nu$ (as in (\ref{eq:metricsE11})) for some $\nu > 0$. Moreover, $i_{\mathfrak s}(E(1, 1), g_\nu) = 1$ and the right-invariant Killing field which at the identity is orthogonal to the abelian ideal generated by $X_1, X_2$ spans the symmetric subspace.
\end{theorem}

\section{Quotient by the leaves of symmetry}

Let $G$ be a $3$-dimensional unimodular Lie group endowed with a left-invariant metric such that $i_{\mathfrak s}(G) = 1$, and let us denote the foliation of symmetry by $\mathcal L$. Assume that the quotient $G/\mathcal L$ admits an $\Iso_0(G)$-invariant metric. If $\dim\Iso_0(G) > 3$, then there exists a Killing field $X \neq 0$ on $G$ acting trivially on $G/\mathcal L$, since $\dim G/\mathcal L = 2$ and so the action of $\Iso_0(G)$ on $G/\mathcal L$ cannot be effective. Such a Killing field is always tangent to $\mathcal L$, or equivalently, $X$ generates the distribution of symmetry, $\mathfrak s = \mathbb RX$. Let us denote by $L(p)$ the leaf of symmetry by $p \in G$, and $\Phi_t$ the flow of $X$. If $u, v \in T_eG$ are orthogonal to the leaf of symmetry at $p$, that is $u, v \in X(p)^\bot$. Then $d\Phi_t(u)$ and $d\Phi_t(v)$ are orthogonal to $L(p)$ at $\Phi_t(p)$ and moreover, $\langle u, v\rangle = \langle d\Phi_t(u), d\Phi_t(v)\rangle$ for all $t$. This means that the projection $\pi: G \to G/\mathcal L$ is a Riemannian submersion. 

\begin{proposition}\label{sec:foliation-of-symmetry}
  If $M$ is one of the Riemannian manifolds $(\SU(2), g_{\lambda, 1, 1})$ with $\lambda > 1$, $(\SU(2), g_{1, 1, \nu})$ with $0 < \nu < 1$, $(\SL(2, \mathbb R), g_{\lambda, 1, 1})$ with $\lambda > 0$, or $H_1$ with a left-invariant metric, then $M \to M/\mathcal L$ is a Riemannian submersion where $\mathcal L$ is the foliation of symmetry and $M/\mathcal L$ has an $\Iso_0(M)$-invariant metric.
\end{proposition}

\begin{remark}
  \begin{enumerate}
  \item It is proved in \cite{berndt-olmos-reggiani-2015} that the families of metrics $g_{\lambda, 1, 1}$ $(\lambda > 1)$ and $g_{1, 1, \nu}$ on $\SU(2)$ correspond to Berger spheres. In this case we have that the Riemannian submersion $\SU(2) \to \SU(2)/\mathcal L \simeq \mathbb S^2$ fibers the round sphere $\mathbb S^2$.
  \item Recall that the metrics $g_{\lambda, 1, 1}$ $(\lambda > 0)$ on $\SL(2, \mathbb R)$ are naturally reductive (with respect to the full isometry group) and in this case we have that the distribution of symmetry coincides with the $\Iso_0(\SL(2, \mathbb R))$-invariant distribution induced by the fixed points of the isotropy. So the foliation of symmetry of $g_{\lambda, 1, 1}$ consists of closed geodesics and the Riemannian submersion $\SL(2, \mathbb R) \to \SL(2, \mathbb R)/\mathcal L \simeq \mathbb H^2$ is a circle bundle over the hyperbolic plane.
  \item It is easy to see that the Riemannian submersion $H_1 \to H_1/\mathcal L$ corresponds to the nontrivial vector bundle $H_1 \to \mathbb R^2$. In fact, one has that the Killing field $X_3$, in the center of $\mathfrak h_1$, generates the distribution of symmetry, and so $X_1, X_2$ span the horizontal distribution. Then by using O'Neill formula we have that the sectional curvature of $H_1/\mathcal L$ vanishes. Recall that, also in this case, the metric on $H_1$ is naturally reductive and the distribution of symmetry is the $\Iso(H_1)$-invariant distribution induced by the fixed points of the isotropy.
  \end{enumerate}
\end{remark}

\begin{proof}
  [{\proofname} of Proposition \ref{sec:foliation-of-symmetry}] It follows from the discussion at the beginning of the section and the above remark.
\end{proof}

Finally, assume that $i_{\mathfrak s}(G) = 1$ and that the isometry group of $G$ is $3$-dimensional. Let $Y \neq 0$ be a left-invariant field such that the associated Killing field $Y^*$ satisfies $(\nabla Y^*)_e = 0$. Then $\mathfrak s_e = \mathbb RY_e$ and, since $\mathfrak s$ is left-invariant, we have $\mathfrak s = \mathbb RY$. Assume that, in addition, the quotient by the foliation of symmetry admits a $G$-invariant metric such that $G \to G/\mathcal L$ is a Riemannian submersion. If $X^*$ is a Killing field such that $X^*|_e$ is orthogonal to $Y_e$, then $X^*$ must have constant length along the integral curve of $Y$ by the identity. By using this observation (see the next remark) we can show that $(\SL(2, \mathbb R), g_{\lambda, \lambda + 1, 1})$ and $(E(1, 1), g_\nu)$ never induce a Riemannian submersion onto their quotients by the leaves of symmetry.

\begin{remark}
  \begin{enumerate}
  \item   Let us consider the case $(\SL(2, \mathbb R), g_{\lambda, \lambda + 1, 1})$. Keeping the notation of the above paragraph and Section \ref{sec:case-sl2}, we can take $Y^* = X_3^*$ and $X^* = X_1^*$. So,
    \begin{align*}
      |X_1^*(\Exp(tX_3)| &  = |\Exp(-tX_3)X_1\Exp(tX_3)|\\
      & = |\cosh(2t)X_1 - \sinh(2t)X_2| \\
      & = \lambda\cosh(2t) + (\lambda + 1)|\sinh(2t)|
    \end{align*}
    is unbounded. 
  \item Analogously for $(E(1, 1), g_\nu)$, we can choose $Y^* = X_3^*$ and $X^* = X_1^*$. In this case $|X_1^*(\Exp(tX_3)| = e^{-t}$ which is also non-constant.
  \end{enumerate}
\end{remark}

\bibliography{/home/silvio/Dropbox/math/bibtex/mybib.bib}
\bibliographystyle{amsalpha}

\end{document}